\documentclass[a4paper, 11pt]{article}
\usepackage{anysize}
\marginsize{3,5cm}{2,5cm}{2,5cm}{2,5cm}
\usepackage{amssymb}
\usepackage{amsmath,amsbsy,amssymb,amscd}
\usepackage{t1enc}
\usepackage[cp1250]{inputenc}
\usepackage[british]{babel}
\usepackage[all]{xy}
\usepackage{color}
\usepackage{amsfonts}
\usepackage{latexsym}
\usepackage{amsthm}
\usepackage{mathrsfs}
\usepackage{hyperref}

\usepackage{color}

\definecolor{orange}{rgb}{1,0.5,0}

\usepackage{mathrsfs} 

\DeclareMathAlphabet{\mathpzc}{OT1}{pzc}{L}{it} 

\def\vep{\varepsilon}

\newtheorem{definition}{Definition}[section]

\newtheorem{proposition}[definition]{Proposition}
\newtheorem{theorem}{Theorem}

\newtheorem{corollary}[definition]{Corollary}
\newtheorem{remark}[definition]{Remark}
\newtheorem{lemma}[definition]{Lemma}

\def\geq{\geqslant}

\def\R{\mathbb{R}}
\def\T{\mathbb{T}}
\newcommand{\raz}{{\mathbf 1}}
\def\eps{\varepsilon}

\def\Z{\mathbb{Z}}
\def\N{\mathbb{N}}

\def\cF{\mathcal F}

\newcommand{\bea}{\begin{eqnarray}}
  \newcommand{\eea}{\end{eqnarray}}
  \newcommand{\beab}{\begin{eqnarray*}}
  \newcommand{\eeab}{\end{eqnarray*}}

  \newcommand{\be}{\begin{equation}}
  \newcommand{\ee}{\end{equation}}

\newcommand{\xbm}{(X,\mathcal{B},\mu)}
\newcommand{\ct}{\mathcal{T}}

\title{Flows with Ratner's property have discrete essential centralizer}
\author{Adam Kanigowski, Mariusz Lema\'nczyk}

\begin{document}
\baselineskip=14pt \maketitle

\begin{abstract}We show that a free, ergodic action of $\R$ with (finite) Ratner's property has countable discrete essential centralizer. Moreover, we show that such flows are mildly mixing.
\end{abstract}

\section{Introduction} When studying dynamics of  horocycle flows acting on the unit tangent spaces of surfaces with constant negative curvature, M.\ Ratner in 1980's discovered a special property, now called Ratner's property, which is a particular  way of divergence of orbits of nearby points.  Namely, the speed of such divergence is polynomial and this seems to be a characteristic property for so called {\em parabolic} dynamics. This property was used by Ratner \cite{Ra2}  to prove some important joining rigidity phenomena in the class of horocycle flows (see also \cite{Ka-Th}, \cite{Th}). Namely, for every ergodic joining, different from product measure, of an ergodic flow $(T_t)$ acting on a probability standard Borel space $\xbm$ with a flow having Ratner's property  the projection map on the $X$-coordinate has finite fibers. Moreover, flows with Ratner's property enjoy pairwise independent property (PID), that is, any self-joining (of arbitrary order) which is pairwise independent is just product measure.
As noticed in \cite{Ju-Ru2}, the PID property  forces  mixing to be mixing of all orders. During the last decade many other flows were shown to enjoy Ratner's property (with some modifications of the original definition but keeping the aforementioned rigidity phenomena).   Mixing of all orders of horocycle flows was known before Ratner's work, see \cite{Ma}, but it is only very recently that, via Ratner's property,  mixing of all orders was established for some  classes of smooth flows on surfaces \cite{Fa-Ka} (see also \cite{Ka-Ku-Ul}). Ratner's property itself however does not imply mixing. Indeed, all other known classes of flows satisfying Ratner's property, see \cite{Fr-Le1}, \cite{Fr-Le3}, \cite{Fr-Le-Le}, \cite{Ka1}, \cite{Ka2}, \cite{Ka-Ku}, are not mixing. On the other hand, all these examples are mildly mixing. As a matter of fact, it was already asked by J.-P.\ Thouvenot in the 1990's  whether Ratner's property is compatible with the rigidity property of flows. More precisely, Thouvenot asked whether we can have a flow $(T_t)_{t\in\R}$ with Ratner's property acting on a probability standard Borel space $\xbm$ for which for some $t_n\to\infty$, we have $f\circ T_{t_n}\to f$ for each $f\in L^2\xbm$. Recall that mild mixing can be be defined as the absence of non-trivial rigid factors, so no one of known examples of flows with Ratner's property was rigid.

In this note, we  will prove the following results.
\begin{theorem}\label{main} Assume that a measurable, measure-preserving flow $\mathcal{T}=(T_t)$ is free, ergodic and enjoys Ratner's property. Then its essential centralizer $EC(\mathcal{T}):=C(\mathcal{T})/\{T_t:\:t\in\R\}$ is discrete and countable.
\end{theorem}

As the essential centralizer of rigid flows is uncountable, this in particular, answers Thouvenot's question.

\begin{theorem}\label{main2} Assume that a measurable, measure-preserving flow $\mathcal{T}=(T_t)$ is free, ergodic and has Ratner's property. Then for each $S\in C(\mathcal{T})$ either $S$ is mildly mixing or $S$ is of finite order.
\end{theorem}
We have the following immediate corollary (which also  answers Thouvenot's question):
\begin{corollary}\label{rat.rig} A flow with  Ratner's property is mildly mixing.
\end{corollary}

\section{Basic definitions}\label{defs} Let $\xbm$ be a probability standard Borel space. By ${\rm Aut}\xbm$ denote the group of all (measure-preserving) automorphisms of that space. Each element $S\in{\rm Aut}\xbm$ can also be treated as a unitary operator on $L^2\xbm$: $Sf:=f\circ S$. Endowed with the strong operator topology, ${\rm Aut}\xbm$ becomes a Polish group.

Throughout, we consider only measurable, $\R$-representations in ${\rm Aut}\xbm$, i.e.\ flows $\mathcal{T}=(T_t)$ on $\xbm$. Measurability means that the map $X\times\R\ni (x,t)\mapsto T_tx$ is measurable. A flow is called {\em rigid} if for a sequence $(t_n)_{n\geq 1}$, $t_n\to +\infty$, we have
$$
\lim_{n\to+\infty }\mu(T_{t_n}A\triangle A)=0 \text{ for every } A\in \mathcal{B}.
$$
This is equivalent to saying that, as operators on $L^2\xbm$, $T_{t_n}\to Id$ strongly.

\begin{remark}\label{pwzbieznosc} \em
If $S_n\in {\rm Aut}\xbm$, $n\geq1$, and $S_n\to Id$ strongly on $L^2\xbm$, then
$S_n\to Id$ in measure. Indeed, no harm to assume that $X$ is a compact metric space with a metric $d$. Fix $\vep>0$ and cover $X$ by finitely many balls $B_1,\ldots,B_m$ o radius $\vep$. Now, given $\delta>0$, for $n\geq n_0$, we have $\mu(B_i\cap S_n^{-1}B_i)\geq(1-\delta)\mu(B_i)$, $i=1,\ldots, m$. Hence $\mu(\{x\in X:\:d(x,S_nx)\geq 2\vep\})<2m\delta$, and the claim follows.

If follows that if $S_n\to Id$ strongly on $L^2\xbm$, then for a subsequence $(n_k)$, we have $S_{n_k}\to Id$ $\mu$-a.e.\end{remark}

Following \cite{Fu-We},  a flow $\mathcal{T}$ is {\em mildly mixing} if it has no non-trivial rigid factors, i.e.
$$
\liminf_{t\to +\infty}\mu(T_{-t}B \triangle B)>0 \text { for every } B\in \mathcal{B}, 0<\mu(B)<1.
$$
Let
$C(\mathcal{T}):=\{S\in {\rm Aut}\xbm\;:\;ST_t=T_tS \text{ for every }t\in\R\}$ denote the {\em centralizer} of $\mathcal{T}$
and let the {\em essential centralizer} of $\mathcal{T}$ be defined as
$$
EC(\mathcal{T})=C(\mathcal{T})/\{T_t\;:\;t\in \R\}.
$$
Then $C(\mathcal{T})$ is closed in ${\rm Aut}\xbm$, hence is a Polish group. Moreover, $EC(\mathcal{T})$ is also a group, although its topological properties depend on whether the subgroup $\{T_t:\:t\in\R\}$ is closed.

\begin{remark}\label{unc}\em It $\mathcal{T}$ is rigid then $EC(\mathcal{T})$ is uncountable. This result is folklore but we provide an argument for completeness.

First notice that in a Polish Abelian group $G$ if we have a dense subgroup $H$ then either $H=G$ or $H$ is a set of first category. Indeed, if $H$ is of second category then $HH^{-1}$ contains a neighborhood $U$ of $1$. Set $F:=\bigcup_{m\geq0}U^m$ to be the group generated (algebraically) by $U$. We have $F\subset H$. Now, $F$ is open (in $G$). But $H$ is a union of cosets of $F$: $H=\bigcup_{h\in H}hF$, so $H$ is also open. But it is also closed (the complement of $H$ is a union of cosets of $H$), so $H$ is a clopen subgroup. Since it is dense, $H=G$.

If $H$ is of first category and if $G=\bigcup_{i\geq1}g_iH$ then $G$ is still of first category which is a contradiction with the fact that $G$ is Polish.

In our context, we use this for $H:=\{T_t:\:t\in\R\}\subset
\overline{\{T_t:\:t\in\R\}}=:G$.\footnote{This argument has been communicated to us by A.\ Danilenko and replaced our first argument based on the open map theorem for Polish groups together with the fact that a solenoidal group is either $\R$, or it is compact, or else it is not locally compact.}
\end{remark}
For joining theory of dynamical systems, we refer the reader to \cite{Gl}.

\subsection{Ratner's property}
Recall now the notion of (finite) Ratner's property introduced in \cite{Fr-Le1}. This is a weakening of the original Ratner's property introduced in \cite{Ra2} in the context of horocycle flows. Assume that $X$ is a $\sigma$-compact metric space with a metric $d$. Let $(T_t)\subset {\rm Aut}\xbm$ be an ergodic flow.

\begin{definition}\label{def:Ratner}\em Fix a finite set $P$ such that $0\notin P$ and $t_0\in \R\setminus\{0\}$. $(T_t)$ is said to have the {\em $R(t_0,P)$-property}  if for every $\eps>0$ and $N\in \N$ there exist $\kappa=\kappa(\eps)>0$, $\delta=\delta(\eps,N)$ and a set $Z=Z(\eps,N)\subset X$, $\mu(Z)>1-\eps$ such that for every $x,y\in Z$, $x$ not in the orbit of $y$ and $d(x,y)<\delta$ there exist $M=M(x,y),L=L(x,y)$, $M,L\geq N$ and $\frac{L}{M}\geq \kappa$ and $p=p(x,y)\in P$ such that
\begin{equation}\label{eqML}
\frac{1}{L}\left\vert\left\{n\in[M,M+L]\;:
\;d(T_{nt_0}x,T_{nt_0+p}y)<\eps\right\}\right\vert>1-\eps.
\end{equation}
\end{definition}
We say that $(T_t)$ has {\em Ratner's property} (with the set $P$) if the set
$$\{s\in \R\:;\; (T_t) \text{ has } R(s,P)-\text{property}\}$$
is uncountable.

\begin{remark}\label{mes:Rat}\em
\begin{enumerate}
\item If $\xbm$ is a probability standard Borel space and there is no good metric structure on $X$ then we say that a flow  $\ct=(T_t)\subset{\rm Aut}\xbm$  has Ratner's property (with the set $P$) if there exists a $\sigma$-compact metric space $(X',d')$ and a flow  $\ct'=(T'_t)\subset{\rm Aut}(X',\mathcal{B}',\mu')$ which has Ratner's property (with the set $P$) and the flows $\ct$ and $\ct'$ are measure-theoretically isomorphic. In such a situation we say that $\ct'$ is a good metric model of $\ct$.
\item It is shown in \cite{Fr-Le1} that Ratner's property does not depend on the choice of a good metric model.
    In particular, Ratner's property does not depend on the choice of metric in Definition~\ref{def:Ratner}.
    It follows that Ratner's property can be defined unambiguously  for flows defined on probability standard Borel spaces.

\item If $(T_t)\subset {\rm Aut}\xbm$ is a free ergodic $\R$-action with $(X,d)$ a $\sigma$-compact metric space and all $T_t$ being isometries then $\ct$  does not have  Ratner's property (e.g.\ consider the flow $T_t(x,y)=(x+t,x+\alpha t)$ on $\T^2$ with $\alpha$ irrational).
\end{enumerate}
\end{remark}

\begin{remark}\label{lin.flow}\em  Notice that if $R_t:\T\to \T$, $t\in\R$,
is the {\em linear} flow on the additive circle, i.e.\ $R_t(x)=x+t\,\mod 1 $, then $(R_t)$ satisfies Ratner's property (with any finite set $P$). Indeed, any two points are in one orbit so Definition~\ref{def:Ratner} holds trivially. The group of eigenvalues of this flow is $\Z$. It is not hard to see that the only ergodic flows which have discrete spectrum and (infinite) cyclic group of eigenvalues are rescalings of the linear flow. Such $\R$-actions are, up to isomorphisms, all ergodic (non-trivial) not free $\R$-actions.\footnote{Note that such flows have the minimal self-joining property.}
\end{remark}

In view of Remark \ref{mes:Rat}, from now on, we assume that a flow $\ct=(T_t)$ acts on probability standard space $\xbm$, where $(X,d)$ is a compact metric space. In view of Remark~\ref{lin.flow}, we will
assume that the $\R$-actions under consideration are free.

\subsection{Proof of Theorem \ref{main}}
We begin with the following lemma:
\begin{lemma}\label{tech}
If $\mathcal{T}=(T_t)$ is ergodic and has Ratner's property then $\{T_t\;:\;t\in \R\}$ is open in $C(\mathcal{T})$.
\end{lemma}
Notice that Theorem \ref{main} is a straightforward consequence of Lemma \ref{tech}:
\begin{proof}[Proof of Theorem \ref{main}]
By Lemma \ref{tech} it follows that the topological group $EC(\mathcal{T})=C(\mathcal{T})/\{T_t\;:\; t\in\R\}$ is discrete (as $\{T_t\;:\; t\in\R\}$ is normal and clopen). But $EC(\mathcal{T})$ is separable (since $C(\mathcal{T})$ is separable). Hence $EC(\mathcal{T})$ is countable.
\end{proof}

Therefore, it is enough to prove Lemma \ref{tech}. Before we provide its proof, more definitions and observations will be needed. Let $\mathcal{T}=(T_t)$ be an ergodic flow such that $T_{t_0}$ is ergodic.  Fix a finite set $P\subset \R\setminus\{0\}$.
Set
\begin{equation}\label{ak}
A_k:=\{x\in X\;:\;d(x,T_{p}x)>k^{-1}\text{ holds  for }p\in P\}.
\end{equation}
We have the following:
\begin{enumerate}
\item[$(P1)$] $\lim_{k\to+\infty}\mu(A_k)=1$.
\item[$(P2)$] There exists $k_0\in \N$ such that for every $\kappa>0$, there exist a set $W_\kappa\subset X$, $\mu(W_\kappa)>0.99$ and a number $N_0\in \N$ such that for every $M,L\geq N_0$, $L/M\geq \kappa$ and every $x\in W_\kappa$, we have
\begin{equation}\label{egorov1}
\frac{1}{L}\big|\{n\in[M,M+L]\;:\; T_{nt_0} x\in A_{k_0}\}\big|\geq 0.99;
\end{equation}
\item[$(P3)$] Let $B\in \mathcal{B}$, $\mu(B)=1$. For every $\kappa>0$ there exist a  set $V_\kappa\subset X$, $\mu(V_\kappa)>0.99$ and a number $N_1\in \N$ such that for every $M,L\geq N_1$, $L/M\geq \kappa$ and every $x\in V_{\kappa}$, we have
\begin{equation}\label{egorov2}
\frac{1}{L}\big|\{n\in[M,M+L]\;:\; T_{nt_0} x\in B\}\big|\geq 0.99.
\end{equation}
\end{enumerate}
Indeed, notice that for every $k\in \N$, we have $A_{k}\subset A_{k+1}$. Let $A:=\bigcup_{k\geq 1}A_k$. Then $A^c$ is the set of periodic points for $\ct$ with periods belonging to $P$, so by ergodicity (of $\ct$), it has measure~$0$. Therefore $\mu(A)=1$ and by $A_{k}\subset A_{k+1}$ for $k\geq 1$, we get $(P1)$.
To prove $(P2)$ we use $(P1)$ to get the existence of $k_0\in \N$ such that $\mu(A_{k_0})\geq 1-10^{-4}$. Then $(P2)$ follows by the pointwise ergodic theorem for $\raz_{A_{k_0}}$ and $T_{t_0}$, and then Egorov's theorem which gives a uniform convergence on a set of arbitrarily large measure (in our case this set is $W_\kappa$).  For $(P3)$, just note that it is a particular case of the same argument used to show $(P2)$ (for a constant sequence of sets).

Assume now that  $(S_n)_{n\in\N}\subset C(\ct)$ and define
\begin{equation}\label{bk}
B_{k}:=\{x\in X\;:\;(\exists_{n_k})\;\; d(x,S_nx)<k^{-1} \text{  for all}\;n\geq n_k\},\;\text{ and }\;\; B:=\bigcap_{k\geq 1}B_k.
\end{equation}
We have the following:
\begin{enumerate}
\item[$(P4)$] If $S_n\to Id$ $\mu$-a.e., then $\mu(B)=1$ (note that if we assume only that $S_n\to Id$ strongly, as operators on $L^2\xbm$, then by Remark~\ref{pwzbieznosc}, we can replace $(S_n)$ by a $\mu$-a.e.\ convergent subsequence).
\item[$(P5)$] If $\{S_n:\:n\in\N\}\cap \{T_t\;:\;t\in \R\}=\emptyset$, then there exists a set $X_0$, $\mu(X_0)=1$, such that for every $n\in \N$ and every $x\in X_0$, we have
\begin{equation}\label{notorb}
S_nx\notin \{T_tx\;:\; t\in \R\}.
\end{equation}
\end{enumerate}
Indeed, for $(P4)$ it is enough to notice that for every $k\geq 1$, $\mu(B_k)=1$. To get $(P5)$ notice that for every $n\in \N$ the set
$$X_n:=\{x\in X\;:\; S_n(x)\in \{T_tx\;:\;t\in\R\}\}$$
is measurable: it is the $X$-projection of the measurable set $\{(x,t)\in X\times\R\;:\; S_nx=T_tx\}$. Moreover, $X_n$ is $\mathcal{T}$-invariant, hence by ergodicity, $\mu(X_n)\in\{0,1\}$. But by assumption, $S_n\notin \{T_t\;:\;t\in \R\}$, hence  $\mu(X_n)=0$. Finally, set $X_0=\cap_{n\geq 1} X_n^c$.

\begin{proof}[Proof of Lemma \ref{tech}]
To show that $\{T_t\;:\; t\in\R\}$ is open it is enough to show that for any sequence $(S_n)_{n\in\N}\subset C(\mathcal{T})$ such that $S_n\to Id$ strongly, there exists $n_0$ such that for $n\geq n_0$, $S_n=T_{t_n}$ for some $t_n\in\R$. We will argue by contradiction. Assume that there exists $(S_n)_{n\in\N}\subset C(\mathcal{T})$ such that $S_n\to Id$ strongly and
$$
\{S_n:\:n\in\N\}\cap \{T_t\;:\; t\in\R\}=\emptyset.
$$

By $(P5)$, after restricting to $X_0$, we may assume that \eqref{notorb} holds for $x\in X$.
Fix  a number $t_0\in \R$ such that $T_{t_0}$ automorphism is ergodic.  We will show that for every such $t_0$, $\mathcal{T}$ does not have the $R(t_0,P)$ property. Since an ergodic flow can have at most countably many non-ergodic time automorphisms, this will finish the proof of Lemma \ref{tech}. Since $t_0\in \R$ is fixed, we will denote the $T_{t_0}$ automorphism  by $T$.

Fix $k_0$ from $(P2)$ and let $0<\epsilon<\min(0.01,k_0^{-2})$ with $\kappa=\kappa(\epsilon)$ coming from the $R(t_0,P)$-property.
Let $N_0$ come from $(P2)$, $N_1$ from $(P3)$ (for $B\in \mathcal{B}$ defined in \eqref{bk}, notice that by $(P4)$, $\mu(B)=1$). Take  $N\geq\max(N_0,N_1)$. Let $\delta=\delta(\epsilon,N)$ and a set $Z=Z(\eps,N)$, $\mu(Z)\geq 1-\eps$ be given by the $R(t_0,P)$-property. Let $k>\max(\delta^{-2},\epsilon^{-2})$ and set $n=n_k$, where $n_k$ is given by~\eqref{bk}. We will show that
for every $x\in B\cap W_\kappa\cap V_\kappa$, $x$ and $S_nx$ do not satisfy~\eqref{eqML}. This means that $(x,S_nx)\notin Z\times Z$. This will finish the proof since by~\eqref{notorb}, $x$ and $S_nx$ are not in the same orbit of $\ct$, by \eqref{bk}, we have $d(x,S_nx)<k^{-1}<\delta^2<\delta$ and
$$
S_n^{-1}(Z)\cap (B\cap W_\kappa\cap V_\kappa\cap Z)=\emptyset,
$$
which yields a contradiction since all the sets in the intersection above have measure at least $0.99$.

Fix $x\in B\cap W_\kappa\cap V_\kappa$. Assume that $(x,S_nx)\in Z\times Z$. Then there exist $M,L\geq N$, $L/M\geq \kappa$ such that~\eqref{eqML} is satisfied for $x, S_nx$. Let (see \eqref{ak}, \eqref{bk})
\begin{equation}\label{def:w}
W:=\{i\in [M,M+L]\;:\; T_i x\in A_{k_0}\cap B\}.
\end{equation}
It follows by \eqref{egorov1} and \eqref{egorov2} that $|W|\geq \frac{3L}{4}$. By \eqref{ak} and \eqref{bk}, it follows that for every $i\in W$, we have for $p\in P$ (by the choice of $k_0$ and $k$)
\begin{multline*}
d(T_{i+p}x,T_{i}(S_nx))\geq d(T_{i+p}x,T_ix)-d(T_ix,T_{i}(S_nx))=\\d(T_{i+p}x,T_ix)-d(T_ix,S_n(T_ix))\geq
k_0^{-1}-k^{-1}\geq \epsilon^{1/2}-\epsilon^2>\eps.
\end{multline*}
Now, since $|W|\geq \frac{3L}{4}$, \eqref{eqML} is not satisfied for $x, S_nx$. This completes the proof.
\end{proof}

\subsection{Proof of Theorem \ref{main2}}

\begin{proposition}\label{prop1} If $\mathcal{T}=(T_t)$ is ergodic, has Ratner's property
then every (non-trivial) factor $\mathcal{S}=(S_t)$ of $\mathcal{T}$ acts freely and has Ratner's property.
\end{proposition}

We will use the following lemma (see \cite{Fr-Le1}, Remark 2.\ and Theorem 5.1.)

\begin{lemma}\label{FL}Let $\ct=(T_t)$ be an ergodic flow on a probability standard Borel space $\xbm$. If $T$ has Ratner's property, then $T$ is a finite extension of each of its non-trivial factors.
\end{lemma}
Now, we prove Proposition \ref{prop1}.
\begin{proof}[Proof of Proposition \ref{prop1}]
Let $\mathcal{S}=(S_t):(Y,\mathcal{C},\nu)\to (Y,\mathcal{C},\nu)$ be a factor of $\ct$, where $(Y,\mathcal{C},\nu)$ is probability standard Borel space, and $(Y,d_Y)$ is compact metric space. In view of Lemma~\ref{FL}, we can assume that $X=Y\times\{0,\ldots,k-1\}$, $\mu$ is the product of $\nu$ and the normalized counting measure on $\{0,\ldots,k-1\}$. Moreover, $X$ becomes a compact metric space with the metric
$d(x,x')=d_Y(y,y')+d_k(i,i')$, where $x=(y,i)$, $x'=(y',i')$ and $d_k$ stands for the discrete metric on $\{0,\ldots,k-1\}$. Let $\pi:X\to Y$ be given by $\pi(y,i)=y$ so that  $\pi\circ T_t=S_t\circ \pi$.

Let us first show that the action of $(S_t)$ is free.  Suppose that the action of $(S_t)$ is not free. Then, for $\nu$-a.e.\ $y\in Y$ there exists $r\in\R$ such that $S_ry=y$. Let $x\in X$ be such that $\pi(x)=y$. Then, for every $n\geq 1$,
we have
$$
\pi(T_{nr}x)=S_{nr}(\pi(x))=S_{nr}(y)=y,
$$
so for every $n\geq 1$,  we have $T_{nr}x\in \pi^{-1}(y)$. But since the action of $(T_t)$ is free, for $\mu$-a.e.\ $x\in X$, we have $T_{n_1s}x\neq T_{n_2s}x$ whenever $n_1\neq n_2$. We would get that $\pi^{-1}y$ is not finite, which is a contradiction.

It remains to show that $\mathcal{S}$ has Ratner's property.
Fix a number $t_0$ such that $(T_{t})$ has the $R(t_0,P)$ property. We will show that $(S_t)$ also has the $R(t_0,P)$ property (with the metric $d_Y$). Fix $\eps, N$ and let $\kappa, \delta, Z$ be as in the definition of $R(t_0,P)$ property for $(T_t)$. If $\eps>0$ is small enough compared to $1/k$, $Z\supset Z_Y\times\{0,\ldots,k-1\}$, where $Z_Y\in \mathcal{C}$ and $\nu(Z_Y)>1-\eps$. Notice that any $x,x'\in Z$ such that $d(x,x')<\delta$ are of the form $x=(y,i), x'=(y',i)$ for some $y,y'\in Z_Y$ and such that $d_Y(y,y')<\delta$. Take any $y,y'\in Z_Y$ such that $y$ not in the orbit of $y'$ and $d_Y(y,y')<\delta$. Then,  let $M,L,p$ be as in the definition of $R(t_0,P)$ for $x=(y,i)$ and $x'=(y',i)$. Then
\begin{multline*}
\frac{1}{L}\left\vert\left\{n\in[M,M+L]\;:\;d_Y(S_{nt_0}y,S_{nt_0+p}y')<\eps\right\}\right\vert\geq\\
\frac{1}{L}\left\vert\left\{n\in[M,M+L]\;:\;d(T_{nt_0}x,T_{nt_0+p}x')<\eps\right\}\right\vert\geq 1-\eps.
\end{multline*}
It follows that $\mathcal{S}$ has the $R(t_0,P)$ property.
This completes the proof of Proposition \ref{prop1}.
\end{proof}

\begin{remark}\label{uw3} \em It follows from \cite{Fr-Le1} that Lemma~\ref{FL} holds without the freeness assumption. By Remark~\ref{lin.flow}, also the assertion of Proposition~\ref{prop1} holds without the assumption that the action of $\mathcal{T}$ is free.
\end{remark}

Now we give the proof of Theorem \ref{main2}:

\begin{proof}[Proof of Theorem \ref{main2}]
Let $ST_t=T_tS$ for all $t\in\R$. Assume first that $S$ is not ergodic.
Notice that $\mathcal{A}:=\{A\in\mathcal{B}: SA=A\}$ is a factor of $(T_t)$. By Lemma \ref{FL} it follows that $(T_t)$ can be represented as a skew product over $Y$ (where $Y=X/\mathcal{A}$) with finite fibers. Let us denote this space by $Y\times\{0,\ldots,k-1\}$ for some $k\in \N$.  Since $S\in C(\mathcal{T})$, it follows that $S$ acts on $Y\times\{0,\ldots,k-1\}$ as the identity on the first coordinate. It is hence of the form $S(x,i)=(x,\tau_x(i))$, with $\tau_x$ being a bijection of $\{0,\ldots,k-1\}$. It easily follows that $S^{k!}=Id$.

Assume now that $S$ is ergodic. We will show that $S$ is mildly mixing in this case. Since $\mathcal{T}$ has Ratner's property, it follows that the Kronecker factor (for $\mathcal{T}$) is trivial. Indeed, if not then by Proposition \ref{prop1} the action on the Kronecker factor is free, so by Remark~\ref{mes:Rat}, it does not have Ratner's property which is a contradiction with Proposition~\ref{prop1}. Thus, $\mathcal{T}$ is weakly mixing. Furthermore, since $S\in C(\mathcal{T})$,  $S$ is weakly mixing. Suppose that $S$ is not mildly mixing and let $\cF$ be a non-trivial rigid factor of it. Then, $S|_{\cF}$ is rigid and $\cF$ is invariant under all $T_t$, whence $\cF$ is a non-trivial factor for $\ct$. Moreover, $S|_{\cF}$ is still weakly mixing, whence all its powers $(S|_{\cF})^n$, $n\in\Z$, are distinct (and in particular $(S|_{\cF})^n\neq Id$ for $n\neq 0$). More than that, if for some $0\neq n\in\Z$, we have $(S|_{\cF})^n=T_{r_n}|_{\cF}$, then $\ct|_{\cF}$ is rigid (since $(S|_{\cF})^n$ is rigid) and hence its essential centralizer would be uncountable which is in contradiction with Theorem~\ref{main} applied to $\ct|_{\cF}$, since the latter action has Ratner's property by Proposition~\ref{prop1}. It follows that
\begin{equation}\label{noteq}
(S|_{\cF})^k\notin \{T_{t}|_{\cF}\;:\; t\in\R\}\text{ for }0\neq k\in\Z.
\end{equation}
Since $S|_{\cF}$ is rigid, the group $\{T_t|_{\cF}:\:t\in\R\}$ is not open in $C(\ct|_{\cF})$. This yields a contradiction with Theorem~\ref{main}. Hence $S$ has no non-trivial rigid factors and so $S$ is mildly mixing. This completes the proof of Theorem \ref{main2}.
\end{proof}

\subsection{Final remarks} We have proved that for the flows satisfying Ratner's property the essential centralizer is countable (and discrete). We would like to emphasize that this fact is not a consequence of the aforementioned in Introduction finite fiber property of ergodic joinings of flows with Ratner's property. Indeed, it seems to be possible to adapt the construction of a rigid and simple automorphism from \cite{Ju-Ru1} to obtain a simple and rigid flow. We have been unable to decide whether
the essential centralizer of a flow with Ratner's property can indeed be infinite (in particular, can it contain an element of infinite order). If
$\mathcal{H}=(h_t)$ is a horocycle flow then $EC(\mathcal{H})$ is finite, see Corollary~4 in  \cite{Ra1}.
For another class of flows enjoying Ratner's property, so called von Neumann special flows over irrational rotations (with the rotation of bounded type) the same phenomenon has been proved \cite{Fr-Le1,Fr-Le2}.

Adam Kanigowski, Department of Mathematics, Penn State University; adkanigowski@gmail.com

Mariusz Lemańczyk, Faculty of Mathematics and Computer Science, Nicolaus Copernic University; mlem@mat.umk.pl

\end{document}